\documentclass[10pt]{llncs}
\usepackage{amssymb}
\usepackage{chessboard}
\begin{document}
\newtheorem{maintheorem}[theorem]{Main Theorem}
\newtheorem{sublemma}{Lemma}[theorem]
\renewcommand{\th}{{\hbox{\scriptsize th}}}
\def\<#1>{\langle#1\rangle}
\newcommand{\Z}{{\mathbb Z}}
\newcommand{\df}{\it}
\newcommand{\set}[1]{\{\,{#1}\,\}}
\newcommand{\text}[1]{\hbox{\rm #1}}
\newcommand{\st}{\mid}
\newcommand{\QEDbox}{\fbox{}}
\newcommand{\OneMove}{\mathop{\rm OneMove}}
\newcommand{\NoKing}{\mathop{\rm NoKing}}
\newcommand{\WhiteMated}{\mathop{\rm WhiteMated}}
\newcommand{\BlackMated}{\mathop{\rm BlackMated}}
\newcommand{\WhiteStalemated}{\mathop{\rm WhiteStalemated}}
\newcommand{\BlackStalemated}{\mathop{\rm BlackStalemated}}
\newcommand{\Ch}{\mathord{\frak{Ch}}}
\newcommand{\WhiteWins}{\mathop{\rm WhiteWins}\nolimits}
\newcommand{\BlackWins}{\mathop{\rm BlackWins}\nolimits}
\newcommand{\Wins}{\mathop{\rm Wins}\nolimits}
\newcommand{\WhiteToPlay}{\mathop{\rm WhiteToPlay}}
\newcommand{\BlackToPlay}{\mathop{\rm BlackToPlay}}
\newcommand{\WhiteInCheck}{\mathop{\rm WhiteInCheck}}
\newcommand{\BlackInCheck}{\mathop{\rm BlackInCheck}}
\newcommand{\Legal}{\mathop{\rm Legal}}
\newcommand{\WhiteMove}{\mathop{\rm WhiteMove}}
\newcommand{\BlackMove}{\mathop{\rm BlackMove}}
\newcommand{\Attack}{\mathop{\rm Attack}}
\newcommand{\Piece}{\mathop{\rm Piece}}
\newcommand{\Move}{\mathop{\rm Move}}

\title{The mate-in-n problem of infinite chess is decidable}
\titlerunning{Mate-in-n for infinite chess is decidable}  
%
%
%
\author{Dan Brumleve\inst{1}
Joel David Hamkins\inst{2}\thanks{Research has been supported in part by grants from the National Science Foundation, the Simons Foundation
and the CUNY Research Foundation.} \and Philipp Schlicht\inst{3}}
\authorrunning{Hamkins and Schlicht} 
\institute{Topsy Labs, Inc.\\
\and New York University\\
Department of Philosophy, 5 Washington Place, New York, NY 10003\\
{\it \&} The City University of New York\\
Mathematics, CUNY Graduate Center, 365 Fifth Avenue, New York, NY 10016\\
Mathematics, College of Staten Island of CUNY, Staten Island, NY 10314\\
\email{jhamkins@gc.cuny.edu}, \texttt{http://jdh.hamkins.org} \and
Universit\"at Bonn\\
Mathematisches Institut, Endenicher Allee 60, 53115 Bonn, Germany\\
\email{schlicht@math.uni-bonn.de} }\maketitle
\begin{abstract}
The mate-in-$n$ problem of infinite chess---chess played on an infinite edgeless board---is the problem of determining whether a designated
player can force a win from a given finite position in at most $n$ moves. Although a straightforward formulation of this problem leads to
assertions of high arithmetic complexity, with $2n$ alternating quantifiers,  the main theorem of this article nevertheless confirms a
conjecture of the second author and C. D. A. Evans by establishing that it is computably decidable, uniformly in the position and in $n$.
Furthermore, there is a computable strategy for optimal play from such mate-in-$n$ positions. The proof proceeds by showing that the
mate-in-$n$ problem is expressible in what we call the first-order structure of chess $\Ch$, which we prove (in the relevant fragment) is
an automatic structure, whose theory is therefore decidable. The structure is also definable in Presburger arithmetic. Unfortunately, this
resolution of the mate-in-$n$ problem does not appear to settle the decidability of the more general winning-position problem, the problem
of determining whether a designated player has a winning strategy from a given position, since a position may admit a winning strategy
without any bound on the number of moves required. This issue is connected with transfinite game values in infinite chess, and the exact
value of the omega one of chess $\omega_1^{\rm chess}$ is not known.
\end{abstract}

Infinite chess is chess played on an infinite edgeless chess board, arranged like the integer lattice $\Z\times\Z$. The familiar chess
pieces---kings, queens, bishops, knights, rooks and pawns---move about according to their usual chess rules, with bishops on diagonals,
rooks on ranks and files and so on, with each player striving to place the opposing king into checkmate. There is no standard starting
configuration in infinite chess, but rather a game proceeds by setting up a particular position on the board and then playing from that
position. In this article, we shall consider only finite positions, with finitely many pieces; nevertheless, the game is sensible for
positions with infinitely many pieces. We came to the problem through Richard Stanley's question on mathoverflow.net
\cite{MO27967Stanley:DecidabilityOfInfiniteChess}.

The {\df mate-in-$n$} problem of infinite chess is the problem of determining for a given finite position whether a designated player can
force a win from that position in at most $n$ moves, counting the moves only of the designated player. For example, figure
\ref{Figure.MateIn12} exhibits an instance of the mate-in-$12$ problem, adapted from a position appearing in
\cite{EvansHamkinsWoodin:TransfiniteGameValuesInInfiniteChess}. We provide a solution at the article's end.
\begin{figure}[h]
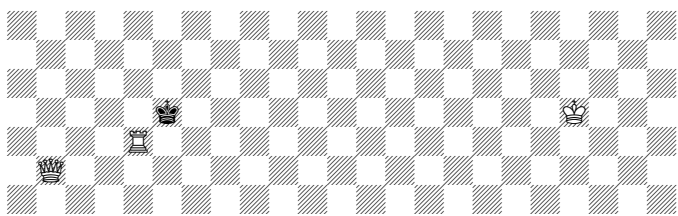
\hskip.8in\chessboard[maxfield=w7,
            boardfontsize=11pt,
            label=false,
            showmover=false,
            borderleft=false,
            borderright=false,
            bordertop=false,
            borderbottom=false,
            margin=false,
            setfen=%
/%
/%
/%
5k13K/%
4R/%
1Q/%
/%
]%
\caption{White to move on an infinite, edgeless board. Can white force mate in 12 moves?}\label{Figure.MateIn12}%
\end{figure}\removelastskip

A naive formulation of the mate-in-$n$ problem yields arithmetical assertions of high arithmetic complexity, with $2n$ alternating
quantifiers: {\it there is a white move, such that for any black reply, there is a counter-play by white}, and so on. In such a
formulation, the problem does not appear to be decidable. One cannot expect to search an infinitely branching game tree even to finite
depth. Nevertheless, the second author of this paper and C. D. A. Evans conjectured that it should be decidable anyway, after observing that
this was the case for small values of $n$, and the main theorem of this paper is to confirm this conjecture.

Before proceeding, let us clarify a few of the rules as they relate to infinite chess. A position should have at most one king of each
color. There is no rule for pawn promotion, as there is no boundary for the pawns to attain. In infinite chess, we abandon as limiting the
50-move rule (asserting that 50 moves without a capture or pawn movement is a draw). A play of the game with infinitely many moves is a
draw. We may therefore abandon the three-fold repetition rule, since any repetition could be repeated endlessly and thus attain a draw, if
both players desire this, and if not, then it needn't have been repeated. This does not affect the answer to any instance of the
mate-in-$n$ problem, because if the opposing player can avoid mate-in-$n$ only because of a repetition, then this is optimal play anyway.
Since we have no official starting rank of pawns, we also abandon the usual initial two-step pawn movement rule and the accompanying {\it
en passant} rule. Similarly, there is no castling in infinite chess.

Chess on arbitrarily large finite boards has been studied, for example in
\cite{FraenkelLichtenstein1981:ComputingPerfecttStrategyForNxNChess}, but the winning positions of infinite chess are not simply those that
win on all sufficiently large finite boards. For example, two connected white rooks can checkmate a black king on any large finite board,
but these pieces have no checkmate position on an infinite board. The lack of edges in infinite chess seems to make it fundamentally
different from large finite chess, although there may be a link between large finite chess and infinite chess on a quarter-infinite board.

Since infinite chess proceeds from an arbitrary position, there are a few additional weird boundary cases that do not arise in ordinary
chess. For example, we may consider positions in which one of the players has no king; such a player could never herself be checkmated, but
may still hope to checkmate her opponent, if her opponent should have a king. We shall not consider positions in which a player has two or
more kings, although one might adopt rules to cover this case. In ordinary chess, one may construe the checkmate winning condition in terms
of the necessary possibility of capturing the king. That is, we may imagine a version of chess in which the goal is simply to capture the
opposing king, with the additional rules that one must do so if possible and one must prevent this possibility for the opponent on the next
move, if this is possible. Such a conception of the game explains why we regard it as illegal to leave one's own king in check and why one
may still checkmate one's opponent with pieces that are pinned to one's own king,\footnote{Even in ordinary finite chess one might argue
that a checkmate position with pinned pieces should count merely as a draw, because once the opposing king is captured, one's own king will
be killed immediately after; but we do not advance this argument.} and gives rise to exactly the same outcomes for the positions of
ordinary chess, with the extra king-capture move simply omitted. This conception directs the resolution of certain other initial positions:
for example, a position with white to move and black in check is already won for white, even if white is in checkmate.

The {\df stalemate-in-$n$} problem is the problem of determining, for a given finite position whether a designated player can force a
stalemate (or a win) in at most $n$ moves, meaning a position from which the player to move has no legal moves. The {\df
draw-in-$n$-by-$k$-repetition} problem is the problem of determining for a given finite position whether a designated player can in at most
$n$ moves force a win or force the opponent into a family of $k$ positions, inside of which he may force his opponent perpetually to
remain. The {\df winning-position} problem is the problem of determining whether a designated player has a winning strategy from a given
finite position. The {\df drawn-position} problem is the problem of determining whether a designated player may force a draw or win from a
given finite position.

\begin{maintheorem}\label{Theorem.MainTheorem}
The mate-in-$n$ problem of infinite chess is decidable. 
\begin{enumerate}%
 \item Moreover, there is a computable strategy for optimal play from a mate-in-$n$ position to achieve the win in the fewest number of
     moves.
 \item Similarly, there is a computable strategy for optimal opposing play from a mate-in-$n$ position, to delay checkmate as long as
     possible. Indeed, there is a computable strategy to enable any player to avoid checkmate for $k$ moves from a given position, if
     this is possible.
 \item In addition, the stalemate-in-$n$ and draw-in-$n$-by-$k$-repetition problems are also decidable, with computable strategies
     giving optimal play.
\end{enumerate}
\end{maintheorem}

Allow us briefly to outline our proof method. We shall describe a certain first order structure $\Ch$, the structure of chess, whose
objects consist of all the various legal finite positions of infinite chess, with various predicates to indicate whether a position shows a
king in check or whether one position is reachable from another by a legal move. The mate-in-$n$ problem is easily expressed as a
$\Sigma_{2n}\vee\Pi_{2n}$ assertion in the first-order language of this structure, the language of chess. In general, of course, there is
little reason to expect such complicated assertions in a first-order structure to be decidable, as generally even the $\Sigma_1$ assertions
about an infinite computable structure are merely computably enumerable and not necessarily decidable. To surmount this fundamental
difficulty, the key idea of our argument is to prove that for any fixed finite collection $A$ of pieces, the reduct substructure $\Ch_A$ of
positions using only at most the pieces of $A$ is not only computable, but in the restricted language of chess is an {\df automatic}
structure, meaning that the domain of objects can be represented as strings forming a regular language, with the predicates of the
restricted language of chess also being regular. Furthermore, the mate-in-$n$ problem for a given position has the same truth value in
$\Ch$ as it does in $\Ch_A$ and is expressible in the restricted language of chess. We may then appeal to the decidability theorem of
automatic structures \cite{KhoussainovNerode1995:AutomaticPresentationsOfStructures,BlumensathGradel2000:AutomaticStructures}, which
asserts that the first order theory of any automatic structure is decidable, to conclude that the mate-in-$n$ problem is decidable. An
alternative proof proceeds by arguing that $\Ch_A$ is interpretable in Presburger arithmetic $\<\mathbb{N},+>$, and this also implies that
the theory of chess with pieces from $A$ is decidable. The same argument methods apply to many other decision problems of infinite chess.
Despite this, the method does not seem to generalize to positions of infinite game value---positions from which white can force a win, but
which have no uniform bound on the number of moves it will take---and as a result, it remains open whether the general winning-position
problem of infinite chess is decidable. Indeed, we do not know even whether the winning-position problem is arithmetic or even whether it
is hyperarithmetic.

One might begin to see that the mate-in-$n$ problem is decidable by observing that for mate-in-$1$, one needn't search through all possible
moves, because if a distant move by a long-range piece gives checkmate, then all sufficiently distant similar moves by that piece also give
checkmate. Ultimately, the key aspects of chess on which our argument relies are: (i) no new pieces enter the board during play, and (ii)
the distance pieces---bishops, rooks and queens---move on straight lines whose equations can be expressed using only addition. Thus, the
structure of chess is closer to Presburger than to Peano arithmetic, and this ultimately is what allows the theory to be decidable.

Let us now describe the first-order structure of chess $\Ch$. Informally, the objects of this structure are all the various finite
positions of infinite chess, each containing all the information necessary to set-up and commence play with an instance of infinite chess,
namely, a finite listing of pieces, their types, their locations and whether they are still in play or captured, and an indication of whose
turn it is. Specifically, define that a {\df piece designation} is a tuple $\<i,j,x,y>$, where $i\in \{{\rm K,k,Q,q,B,b,N,n,R,r,P,p}\}$ is
the piece type, standing for king, queen, bishop, knight, rook or pawn, with upper case for white and lower case for black; $j$ is a binary
indicator of whether the piece is in play or captured; and $(x,y)\in\Z\times\Z$ are the coordinates of the location of the piece on the
chessboard (or a default value when the piece is captured). A finite {\df position} is simply a finite sequence of piece designations,
containing at most one king of each color and with no two live pieces occupying the same square on the board, together with a binary
indicator of whose turn it is to play next.
One could easily define an equivalence relation on the positions for when they correspond to the same set-up on the board---for example,
extra captured bishops do not matter, and neither does permuting the piece designations---but we shall actually have no need for that
quotient. Let us denote by $\text{Ch}$ the set of all finite positions of infinite chess. This is the domain of what we call the structure
of chess $\Ch$.

We shall next place predicates and relations on this structure in order to indicate various chess features of the positions. For example,
$\WhiteToPlay(p)$ holds when position $p$ indicates that it is white's turn, and similarly for $\BlackToPlay(p)$. The relation
$\OneMove(p,q)$ holds when there is a legal move transforming position $p$ into position $q$. We adopt the pedantic formalism for this
relation that the representation of the pieces in $p$ and $q$ is respected: the order of listing the pieces is preserved and captured
pieces do not disappear, but are marked as captured. The relation $\BlackInCheck(p)$ holds when $p$ shows the black king to be in check,
and similarly for $\WhiteInCheck(p)$. We define $\BlackMated(p)$ to hold when it is black's turn to play in $p$, black is in check, but
black has no legal move; the dual relation $\WhiteMated(p)$ when it is white to play, white is in check, but has no legal move. Similarly,
$\BlackStalemated(p)$ holds when it is black's turn to play, black is not in check, but black has no legal move; and similarly for the dual
$\WhiteStalemated(p)$. The {\df structure of chess} $\Ch$ is the first-order structure with domain $\text{Ch}$ and with all the relations
we have mentioned here. The language is partly redundant, in that several of the predicates are definable from the others, so let us refer
to the language with only the relations $\WhiteToPlay$, $\OneMove$, $\BlackInCheck$ and $\WhiteInCheck$ as the restricted language of
chess. Later, we shall also consider expansions of the language.

Since the $\OneMove(p,q)$ relation respects the order in which the pieces are enumerated, the structure of chess $\Ch$ naturally breaks
into the disjoint components $\Ch_A$, consisting of those positions whose pieces come from a {\df piece specification} type $A$, that is, a
finite list of chess-piece types. For example, $\Ch_{\rm KQQkb}$ consists of the chess positions corresponding to the white king and two
queens versus black king and one bishop problems, enumerated in the {\rm KQQkb} order, with perhaps some of these pieces already captured.
Since there is no pawn promotion in infinite chess or any other way to introduce new pieces to the board during play, any game of infinite
chess beginning from a position with piece specification $A$ continues to have piece specification $A$, and so chess questions about a
position $p$ with type $A$ are answerable in the substructure $\Ch_A$. We consider the structure $\Ch_A$ to have only the restricted
language of chess, and so it is a reduct substructure rather than a substructure of $\Ch$.

We claim that the mate-in-$n$ problem of infinite chess is expressible in the structure of chess $\Ch$. Specifically, for any finite list
$A$ of chess-piece types, there are assertions $\WhiteWins_n(p)$ and $\BlackWins_n(p)$ in the restricted language of chess, such that for
any position $p$ of type $A$, the assertions are true in $\Ch_A$ if and only if that player has a strategy to force a win from position $p$
in at most $n$ moves. This can be seen by a simple inductive argument. For the $n=0$ case, and using the boundary-case conventions we
mentioned earlier, we define $\WhiteWins_0(p)$ if it is black to play, black is in checkmate and white is not in check, or it is white to
play and black is in check. Next, we recursively define $\WhiteWins_{n+1}(p)$ if either white can win in $n$ moves, or it is white's turn
to play and white can play to a position from which white can win in at most $n$ moves, or it is black's turn to play and black indeed has
a move (so it is not stalemate), but no matter how black plays, white can play to a position from which white can win in at most $n$ moves.
It is clear by induction that these assertions exactly express the required winning conditions and have complexity
$\Sigma_{2n}\vee\Pi_{2n}$ in the language of chess (since perhaps the position has the opposing player going first), or complexity
$\Sigma_{2n+1}\vee\Pi_{2n+1}$ in the restricted language of chess, since to define the checkmate condition from the in-check relation adds
an additional quantifier.

\begin{lemma}\label{Lemma.Ch_AIsAutomatic}
For any finite list $A$ of chess-piece types, the structure $\Ch_A$ is automatic.
\end{lemma}

\begin{proof}
This crucial lemma is the heart of our argument. What it means for a structure to be automatic is that it can be represented as a
collection of strings forming a regular language, and with the predicates also forming regular languages. The functions are replaced with
their graphs and handled as predicates. We refer the reader to \cite{KhoussainovMinnes2010:ThreeLecturesOnAutomaticStructures} for further
information about automatic structures. We shall use the characterization of regular languages as those consisting of the set of strings
that are acceptable to a read-only Turing machine. We shall represent integers $x$ and $y$ with their signed binary representation,
consisting of a sign bit, followed by the binary representation of the absolute value of the integer. Addition and subtraction are
recognized by read-only Turing machines.

Let us discuss how we shall represent the positions as strings of symbols. Fix the finite list $A$ of chess-piece types, and consider all
positions $p$ of type $A$. Let $N$ be the length of $A$, that is, the number of pieces appearing in such positions $p$. We shall represent
the position $p$ using $3N+1$ many strings, with end-of-string markers and padding to make them have equal length. One of the strings shall
be used for the turn indicator. The remaining $3N$ strings shall represent each of the pieces, using three strings for each piece. Recall
that each piece designation of $p$ consists of a tuple $\<i,j,x,y>$, where $i$ is the piece type, $j$ is a binary indicator for whether the
piece is in play or captured and $(x,y)$ are the coordinates of the location of the piece in $\Z\times\Z$. Since we have fixed the list $A$
of piece types, we no longer need the indicator $i$, since the $k^\th$ piece will have the type of the $k^\th$ symbol of $A$. We represent
$j$ with a string consisting of a single bit, and we represent the integers $x$ and $y$ with their signed binary representation. Thus, the
position altogether consists of $3N+1$ many strings. Officially, one pads the strings to the same length and interleaves them together into
one long string, although for regularity one may skip this interleaving step and simply work with the strings on separate tapes of a
multi-tape machine. Thus, every position $p\in\text{Ch}_A$ is represented by a unique sequence of $3N+1$ many strings. We now argue that
the collection of such sequences of strings arising from positions is regular, by recognizing with a read-only multi-tape Turing machine
that they obey the necessary requirements: the turn indicator and the alive/captured indicators are just one bit long; the binary
representation of the locations has the proper form, with no leading zeros (except for the number $0$ itself); the captured pieces display
the correct default location information; and no two live pieces occupy the same square. If all these tests are passed, then the input does
indeed represent a position in $\Ch$.

Next, we argue that the various relations in the language of chess are regular. For this, it will be helpful to introduce a few more
predicates on the collection of strings. From the string representing a position $p$, we can directly extract from it the string
representing the location of the $i^\th$ piece in that position. Similarly, the coding of whose turn it is to play is explicitly given in
the representation, so the relations $\WhiteToPlay(p)$ and $\BlackToPlay(p)$ are regular languages.

Consider the relation $\Attack_i(p,x,y)$, which holds when the $i^\th$ piece of the position represented by $p$ is attacking the square
located at $(x,y)$, represented with signed binary notation, where by attack we mean that piece $i$ could move so as to capture an opposing
piece on that square, irrespective of whether there is currently a friendly piece occupying that square or whether piece $i$ could not in
fact legally move to that square because of a pin condition. We claim that the $\Attack$ relation is a regular language. It suffices to
consider the various piece types in turn, once we know the piece is alive. If the $i^\th$ piece is a king, we check that its location in
$p$ is adjacent to $(x,y)$, which is a regular language because $\set{(c,d)\st |c-d|=1}$, where $c$ and $d$ are binary sequences
representing signed integers, is regular. Similarly, the attack relation in the case of pawns and knights is also easily recognizable by a
read-only Turing machine. Note that for bishops, rooks and queens, other pieces may obstruct an attack. Bishops move on diagonals, and two
locations $(x_0,y_0)$ and $(x_1,y_1)$ lie on the same diagonal if and only if they have the same sum $x_0+y_0=x_1+y_1$ or the same
difference $x_0-y_0=x_1-y_1$, which is equivalent to $x_0+y_1=x_1+y_0$, and since signed binary addition is regular, this is a regular
condition. When two locations $(x_0,y_0)$ and $(x_1,y_1)$ lie on the same diagonal, then a third location $(a,b)$ obstructs the line
connecting them if and only if it also has that sum $a+b=x_0+y_0=x_1+y_1$ or difference $a-b=x_0-y_0=x_1-y_1$, and also has $x_0<a<x_1$ or
$x_1<a<x_0$. This is a regular condition on the six variables $(x_0,y_0,x_1,y_1,a,b)$, because it can be verified by a read-only Turing
machine. The order relation $x<y$ is a regular requirement on pairs $(x,y)$, because one can check it by looking at the signs and the first
bit of difference. So the attack relation for bishops is regular. Rooks move parallel to the coordinate axes, and so a rook at $(x_0,y_0)$
attacks the square at $(x_1,y_1)$, if it is alive and these two squares are different but lie either on the same rank $y_0=y_1$, or on the
same file $x_0=x_1$, and there is no obstructing piece. This is clearly a condition that can be checked by a read-only Turing machine, and
so it is a regular requirement. Finally, the attack relation for queens is regular, since it reduces to the bishop and rook attack
relations.

It follows now that the relation $\WhiteInCheck(p)$, which holds when the position shows the white king in check, is also regular. With a
read-only Turing machine we can clearly recognize whether the white king is indicated as alive in $p$, and then we simply take the
disjunction over each of the black pieces listed in $A$, as to whether that piece attacks the location square of the white king. Since the
regular languages are closed under finite unions, it follows that this relation is regular. Similarly the dual relation $\BlackInCheck(p)$
is regular.

Consider now the $\Move_i(p,x,y)$ relation, similar to the attack relation, but which holds when the $i^\th$ piece in position $p$ is alive
and may legally move to the square $(x,y)$. It should be the correct player's turn; there should be no obstructing pieces; the square at
$(x,y)$ should not be occupied by any friendly piece; and the resulting position should not leave the moving player in check. Each of these
conditions can be verified as above. A minor complication is presented by the case of pawn movement, since pawns move differently when
capturing than when not capturing, and so for pawns one must check whether there is an opposing piece at $(x,y)$ if the pawn should move
via capture.

Consider next the relation $\OneMove_i(p,q)$, which holds when position $p$ is transformed to position $q$ by a legal move of piece $i$.
With a read-only Turing machine, we can easily check that position $p$ indicates that it is the correct player's turn, and by the relations
above, that piece $i$ may legally move to its location in $q$, that any opposing piece occupying that square in $p$ is marked as captured
in $q$, and that none of the other pieces of $p$ are moved or have their capture status modified in $q$. Thus, this relation is a regular
language.

Finally, we consider the relation $\OneMove(p,q)$ in $\Ch_A$, which holds when position $p$ is transformed to position $q$ by a single
legal move. This  is simply the disjunction of $\OneMove_i(p,q)$ for each piece $i$ in $A$, and is therefore regular, since the regular
languages are closed under finite unions.

Thus, we have established that the domain of $\Ch_A$ is regular and all the predicates in the restricted language of chess are regular, and
so the structure $\Ch_A$ is automatic, establishing the lemma. \hfill\QEDbox
\end{proof}

We now complete the proof of the main theorem. Since the structure $\Ch_A$ is automatic, it follows by the decidability theorem of
automatic structures \cite{KhoussainovNerode1995:AutomaticPresentationsOfStructures,BlumensathGradel2000:AutomaticStructures} that the
first-order theory of this structure is uniformly decidable. In particular, since the mate-in-$n$ question is expressible in this
structure---and we may freely add constant parameters---it follows that the mate-in-$n$ question is uniformly decidable: there is a
computable algorithm, which given a position $p$ and a natural number $n$, determines yes-or-no whether a designated player can force a win
in at most $n$ moves from position $p$. Furthermore, in the case that the position $p$ is mate-in-$n$ for the designated player, then there
is a computable strategy providing optimal play: the designated player need only first find the smallest value of $n$ for which the
position is mate-in-$n$, and then search for any move leading to a mate-in-$(n-1)$ position. This value-reducing strategy achieves victory
in the fewest possible number of moves from any finite-value position. Conversely, there is a computable strategy for the losing player
from a mate-in-$n$ position to avoid inadvertantly reducing the value on a given move, and thereby delay the checkmate as long as possible.
Indeed, if a given position $p$ is not mate-in-$n$, then we may computably find moves for the opposing player that do not inadvertantly
result in a mate-in-$n$ position. Finally, we observe that the stalemate-in-$n$ and draw-in-$n$-by-$k$-repetition problems are similarly
expressible in the structure $\Ch_A$, and so these are also decidable and admit computable strategies to carry them out. This completes the
proof of the main theorem. \hfill \QEDbox

\medskip

An essentially similar argument shows that the structure of chess $\Ch_A$, for any piece specification $A$, is definable in Presburger
arithmetic $\<\mathbb{N},+>$. Specifically, one codes a position with a fixed length sequence of natural numbers, where each piece is
represented by a sequence of numbers indicating its type, whether it is still in play, and its location (using extra numbers for the sign
bits). The point is that the details of our arguments in the proof of the main theorem show that the attack relation and the other
relations of the structure of chess are each definable from this information in Presburger arithmetic. Since Presburger arithmetic is
decidable, it follows that the theory of $\Ch_A$ is also decidable.

We should like to emphasize that our main theorem does not appear to settle the decidability of the winning-position problem, the problem
of determining whether a designated player has a winning strategy from a given position. The point is that a player may have a winning
strategy from a position, without there being any finite bound on the number of moves required. Black might be able to delay checkmate any
desired finite amount, even if every play ends in his loss, and there are positions known to be winning but not mate-in-$n$ for any $n$.
These are precisely the positions with infinite game value in the recursive assignment of ordinal values to winning positions:  already-won
positions for white have value $0$; if a position is white-to-play, the value is the minimum of the values of the positions to which white
may play, plus one; if it is black-to-play, and all legal plays have a value, then the value is the supremum of these values. The winning
positions are precisely those with an ordinal value, and this value is a measure of the distance to a win. A mate-in-$n$ position, with $n$
minimal, has value $n$. A position with value $\omega$ has black to play, but any move by black will be mate-in-$n$ for white for some $n$,
and these are unbounded. The omega one of chess, denoted $\omega_1^{\rm chess}$, is defined in
\cite{EvansHamkinsWoodin:TransfiniteGameValuesInInfiniteChess} to be the supremum of the values of the finite positions of infinite chess.
The exact value of this ordinal is an open question, although an argument of Blass appearing in
\cite{EvansHamkinsWoodin:TransfiniteGameValuesInInfiniteChess} establishes that $\omega_1^{\rm chess}\leq\omega_1^{ck}$, as well as the
accompanying fact that if a player can win from position $p$, then there is a winning strategy of hyperarithmetic complexity.  Although we
have proved that the mate-in-$n$ problem is decidable, we conjecture that the general winning-position problem is undecidable and indeed,
not even arithmetic.

Consider briefly the case of three-dimensional infinite chess, as well as higher-dimensional infinite chess. Variants of three-dimensional
(finite) chess arose in the late nineteenth century and have natural infinitary analogues. Without elaborating on the details---there are
various reasonable but incompatible rules for piece movement---we remark that the method of proof of our main theorem works equally well in
higher dimensions.
\begin{corollary}\label{Corollary.MateInNforHigherDimensions}
The mate-in-$n$ problem for $k$-dimensional infinite chess is decidable.
\end{corollary}

Results in \cite{EvansHamkinsWoodin:TransfiniteGameValuesInInfiniteChess} establish that the omega one of infinite positions in
three-dimensional infinite chess is exactly true $\omega_1$; that is, every countable ordinal arises as the game value of an infinite
position of three-dimensional infinite chess.

We conclude the article with a solution to the chess problem we posed in figure \ref{Figure.MateIn12}. With white Qe5+, the black king is
forced to the right kg4, and then white proceeds in combination Rg3+ kh4 Qg5+ ki4, rolling the black king towards the white king, where
checkmate is delivered on white's $13^\th$ move. Alternatively, after Qe5+ kg4, white may instead play Ks4, moving his king to the left,
forcing the two kings together from that side, and this also leads to checkmate by the queen on the $13^{th}$ move. It is not possible for
white to checkmate in 12 moves, because if the two kings do not share an adjacent square, there is no checkmate position with a king, queen
and rook versus a king. Thus, white must force the two kings together, and this will take at least 12 moves, after which the checkmate move
can now be delivered, meaning at least 13 moves.\footnote{C. D. A. Evans (US national master) confirms this mate-in-13-but-not-12
analysis.} However, white can force a stalemate in 12 moves, by moving Qe5+, and then afterwards moving only the white king towards the
black king, achieving stalemate on the 12th move, as the black king is trapped on f4 with no legal move. White can force a draw by
repetition in $3$ moves, by trapping the black king in a $4\times 4$ box with the white queen and rook at opposite corners, via  Qe5+ kg4
Ri3 kh4 Qf6+, which then constrains the black king to two squares.

\bibliographystyle{alpha}
\bibliography{MathBiblio,HamkinsBiblio}

\end{document}